\documentclass[11pt,a4paper,reqno,oneside]{amsart}

\usepackage{latexsym}
\usepackage{amssymb}
\usepackage{amsfonts}
\usepackage{graphics}
\usepackage{color}
\usepackage{mathcomp}

\author{Thomas Leuther} 
\address[Thomas Leuther]{University of Li\`ege, Department of Mathematics\\
 Grande Traverse, 12 - B37, B-4000 Li\`ege, Belgium}
\email{Thomas.Leuther[at]ulg.ac.be}
\author{Pierre Mathonet}
\address[Pierre Mathonet]{Mathematics Research Unit, FSTC, University of Luxembourg \\
6, rue Coudenhove-Kalergi, L-1359 Luxembourg, Luxembourg}
\email{pierre.mathonet[at]uni.lu,P.Mathonet[at]ulg.ac.be}
\author{Fabian Radoux} 
\address[Fabian Radoux]{University of Li\`ege, Department of mathematics \\
 Grande Traverse, 12 - B37, B-4000 Li\`ege, Belgium}
\email{Fabian.Radoux[at]ulg.ac.be}

\date{\today} 
\title[Equivariant quantizations]{On $\mathfrak{osp}(p+1,q+1|2r)$-equivariant quantizations}

\theoremstyle{plain}
\newtheorem{theorem}{Theorem}[section]
\newtheorem{lemma}[theorem]{Lemma}
\newtheorem{proposition}[theorem]{Proposition}

\theoremstyle{definition}
\newtheorem{definition}[theorem]{Definition}

\theoremstyle{remark}

\newtheorem{remark}[theorem]{Remark}

\newcommand{\gl}{\mathfrak{gl}}
\newcommand{\str}{\mathrm{str}}
\newcommand{\Id}{\mathrm{Id}}

\newcommand{\R}{\mathbb{R}}

\newcommand{\N}{\mathbb{N}}

\renewcommand{\L}{\mathcal{L}}

\renewcommand{\S}{\mathcal{S}}
\newcommand{\F}{\mathcal{F}}
\newcommand{\g}{\mathfrak{g}}

\newcommand{\h}{\mathfrak{h}}

\newcommand{\dive}{\mathrm{div}}
\newcommand{\G}{\mathrm{G}}
\newcommand{\euler}{\mathcal{E}}

\newcommand{\aff}{\mathrm{Aff}}

\newcommand{\Vect}{\mathrm{Vect}}

 \renewcommand{\div}{\mathrm{div}}

\newcommand{\cc}{{\mathcal{C}}}

\begin{document}

\begin{abstract}
We investigate the concept of equivariant quantization over the superspace $\R^{p+q|2r}$, with respect to the orthosymplectic algebra $\mathfrak{osp}(p+1,q+1|2r)$. 
Our methods and results vary upon the superdimension $p+q-2r$.
When the superdimension is nonzero, we manage to obtain a result which is similar to the classical theorem of Duval, Lecomte and Ovsienko: we prove the existence and uniqueness of the equivariant quantization except in some resonant situations. 
To do so, we have to adapt their methods to take into account the fact that the Casimir operator of the orthosymplectic algebra on supersymmetric tensors is not always diagonalizable, when the superdimension is negative and even. 
When the superdimension is zero, the situation is always resonant, but we can show the existence of a one-parameter family of equivariant quantizations for symbols of degree at most two.
\end{abstract}

\maketitle



\section{Introduction}

The concept of \emph{projectively equivariant quantization} over $\R^n$ was introduced by P. Lecomte and V. Ovsienko in \cite{LecOvs99}. 
If we denote by $\mathfrak{pgl}(n+1)$ the Lie algebra of infinitesimal projective transformations of $\R^n$, such a quantization is an isomorphism of $\mathfrak{pgl}(n+1)$-modules from the space $\S_\delta(\R^n)$ of contravariant symmetric tensors with coefficients in $\delta$-densities to the space ${\mathcal D}_{\lambda,\lambda+\delta}(\R^n)$ of differential operators mapping $\lambda$-densities to $\lambda+\delta$-densities. 
This isomorphism is moreover required to satisfy a natural normalization condition (see (\ref{norma})).
 
In \cite{LecOvs99}, P. Lecomte and V. Ovsienko showed the existence and uniqueness of the projectively equivariant quantization in the case $\delta=0$.
This result was generalized in \cite{DuvOvs01,Lecras} for arbitrary $\delta\in \R\setminus C$, where $C$ is a set of \emph{critical values}.
In \cite{DuvLecOvs99}, C. Duval, P. Lecomte and V. Ovsienko investigated the concept of \emph{conformally equivariant quantization}. 
It is a quantization that intertwines the actions of the algebra $\mathfrak{so}(p+1,q+1)$ of infinitesimal conformal transformations of $\R^{p+q}$. 
The main result in this case is the existence and uniqueness of such a quantization, provided the shift value $\delta$ is not critical (see also \cite{Mic11,Sil09} for a refined classification).

Various generalizations of these results were considered in recent years. 
For instance, in \cite{Bou00,Bou01,DuvOvs01conf,Lou03,Leconj,Bor00,Sil09}, the concepts of projectively and conformally equivariant quantizations were extended to arbitrary manifolds endowed with projective or conformal structures. 
Also, other spaces of differential operators were considered in \cite{Han04,MatRad05,Fox05,Han07,MatRad07,MatRad09,MatRad10,CapSil10}.

Recently, several papers dealt with the problem of equivariant quantizations in the context of supergeometry. 
Let us quote for instance \cite{GarMelOvs07} and \cite{Mel09} for equivariant quantizations over the supercircle $S^{1|1}$ and $S^{1|2}$ and \cite{Mic09} for equivariant quantizations over supercotangent bundles. 
The extension of the theory of projectively equivariant quantizations to the framework of the superspace $\R^{n|m}$ or more generally to supermanifolds endowed with a projective class of superconnections was investigated in \cite{Geo09,MatRad11,LeuRad11}. 

It is then natural to consider the extension to supergeometry of the conformally equivariant quantization in the sense of \cite{DuvLecOvs99}. 
We consider the orthosymplectic algebra $\mathfrak{osp}(p+1,q+1|2r)$ that we realize as a subalgebra of vector fields over the superspace $\R^{p+q|2r}$ and we analyze the existence of the equivariant quantizations with respect to this algebra. 

As in the projective setting \cite{MatRad11}, the situation depends on the superdimension of the space under consideration. 
\begin{itemize}
\item 
When the superdimension is nonzero, we adapt the methods of \cite{DuvLecOvs99} to take into account that the Casimir operator of the orthosymplectic algebra on supersymmetric tensors is not always semi-simple. Using generalized eigenvectors of Casimir operators, we thus manage to prove existence and uniqueness of the quantization except in a countable set of resonant values of the shift $\delta$. 
\item 
When the superdimension is zero, the results do not depend on the shift $\delta$. Actually, the situation is special since any value of $\delta$ resonant. Nevertheless, we prove in this case the existence of a one-parameter family of quantizations for symbols of degree at most two.
\end{itemize}
Finally we provide explicit formulae for the quantization of symbols of degree at most two.

\section{Notation and problem setting}\label{tens}

In this section, we will recall the definitions of the spaces of differential operators acting on densities and of their corresponding spaces of symbols over the superspace $\R^{p+q|2r}$. We refer the reader to \cite{MatRad11} and references therein for a more detailed discussion of these concepts.
Then we will set the problem of existence of $\mathfrak{osp}(p+1,q+1|2r)$-equivariant quantizations. 

We denote by $d$ the superdimension of the superspace $\R^{p+q|2r}$, that is $d=p+q-2r$.
Throughout the paper, we only consider Lie superalgebras over $\R$. 
For any natural number $n$ we set $I_n=\{1,\ldots,n\}$. 
We denote by $\tilde{a}$ the parity of a homogeneous object $a$. For indices $i\in I_{p+q+2r}$, we set $\tilde{i}=0$ if $i\leqslant p+q$ and $\tilde{i}=1$ otherwise. We denote by $x^1,\ldots,x^{p+q}$ the set of even indeterminates and by $\theta^1,\ldots,\theta^{2r}$ the set of odd indeterminates. 
We also use the unified notation $y^i$, where $i\in I_{p+q+2r}$, for the set of even and odd indeterminates, $y^1,\ldots,y^{p+q}$ being the even $x^1,\ldots,x^{p+q}$. 
 
For $i\in I_{p+q+2r}$, we denote by $e_i$ (resp. $\varepsilon^i$) the column (resp. row) vector in $\R^{p+q|2r}$ (resp. $(\R^{p+q|2r})^*$) whose entries are 0, except the $i$-th one which is 1. 
We denote by $\mathcal{B}_{p+q|2r}$ the basis $(e_1,\ldots,e_{p+q+2r})$, and by $\mathcal{B}^*_{p+q|2r}$ the corresponding basis in $(\R^{p+q|2r})^*$.

\subsection{Densities and weighted symmetric tensors}
The $\Vect(\R^{p+q|2r})$-modules of densities and of weighted symmetric tensor fields are both particular cases of the general construction of a Lie derivative on generalized tensors considered for instance in \cite{BerLei81Ser,GroLei02}. 
For every representation $(V,\rho)$ of $\gl(p+q|2r)$, the space of \emph{tensor fields of type $(V,\rho)$} is defined as $T(V)=\F\otimes V$, where $\F = \mathrm{C}^{\infty}(\R^{p+q|2r})$, and the Lie derivative of a tensor field $f\otimes v$ along a vector field $X \in \Vect(\R^{p+q|2r})$ is defined by the formula
\begin{equation}\label{eqRadouxLeites}
{\rm L}_{X}(f \otimes v)=X(f) \otimes v + (-1)^{\widetilde{X}\tilde{f}}\sum_{ij}f J_{i}^{j}\otimes\rho(e_{j}^{i})v ,
\end{equation}
where $J_{i}^{j}=(-1)^{\tilde{y^i}\widetilde{X}+1}(\partial_{y^{i}}X^{j})$ and $e_{j}^{i}$ is the operator defined in the basis $\mathcal{B}_{p+q|2r}$ by $e_{j}^{i}(e_k)=\delta^i_ke_j$. 

The particular case of densities corresponds to a vector space $V = B^\lambda$ of dimension $1|0$ spanned by one element $u$ and the representation of $\gl(p+q|2r)$ defined by $\rho(A)u=-\lambda\,\str(A)u$. 
Note that in the formulas below, we will not write down explicitly the generator $u$ of $B^\lambda$ unless this leads to confusion.
\begin{definition}
The $\Vect(\R^{p+q|2r})$-module $\mathcal{F}_{\lambda}$ of densities of degree $\lambda$ over $\R^{p+q|2r}$ is the space $T(B^\lambda)$ endowed with the action of $\Vect(\R^{p+q|2r})$ defined by formula (\ref{eqRadouxLeites}).
\end{definition}
The space $\mathcal{F}_{\lambda}$ is thus isomorphic to the space of functions $\mathcal{F}$ endowed with the Lie derivative
\begin{equation}\label{formula:Lie density}
{\rm L}_X^\lambda f=X(f)+\lambda\,\dive(X)f ,
\end{equation}
where the divergence of the vector field $X=\sum_{i=1}^{p+q+2r}X^i\partial_{y^i}$ is given by
\[
\dive(X)=\sum_{i=1}^{p+q+2r}(-1)^{\tilde{y_i}\widetilde{X^i}}\partial_{y^i}X^i .
\]

We denote by $S^k$ the space  of supersymmetric tensors of degree $k$ over $\R^{p+q|2r}$. Recall that the symmetric tensor product of homogeneous elements $v_1,\ldots,v_k\in \R^{p+q|2r}$ is defined by
\[
v_1\vee\cdots\vee v_k=\sum_{\sigma\in \mathfrak{S}_k}\mathrm{sgn}(\sigma,v_1,\ldots,v_k)v_{\sigma^{-1}(1)}\otimes\cdots\otimes v_{\sigma^{-1}(k)} ,
\]
where $\mathrm{sgn}(\sigma,v_1,\ldots,v_k)$ is the sign of the permutation $\sigma'$ induced by $\sigma$ on the ordered subset of all odd elements among $v_{1},\ldots,v_{k}$. 
There is a natural representation $\rho$ of $\gl(p+q|2r)$ on $S^k$ defined by 
\begin{equation}\label{rhotens}
\rho(A)(v_{1}\vee\cdots\vee
v_{k})=\sum_{i=1}^{k}(-1)^{\tilde{A}(\sum_{l=1}^{i-1}\tilde{v}_{l})}v_{1}\vee\cdots\vee
A v_{i}\vee\cdots\vee v_{k} .
\end{equation}
The space $S^k_\delta$ $(\delta\in\R)$ of \emph{weighted symmetric tensors} is the tensor product $B^\delta\otimes S^k$ endowed with the tensor product representation of $\gl(p+q|2r)$. 
\begin{definition}\label{defsymb}
The space $\S^k_{\delta}$ of weighted symmetric tensor fields is the space $T(S^k_\delta)=\F \otimes S^k_\delta$ endowed with the action of $\Vect(\R^{p+q|2r})$ defined by formula (\ref{eqRadouxLeites}).
\end{definition}

\subsection{Differential operators and symbols}
The space $\mathcal{D}_{\lambda,\mu}$ of linear differential operators from $\mathcal{F}_\lambda$ to $\mathcal{F}_\mu$ is filtered by the order of differential operators. We denote by $\mathcal{D}^k_{\lambda,\mu}$ the space of differential operators of order at most $k$. Any differential operator $D\in \mathcal{D}^k_{\lambda,\mu}$ can be written as
\begin{equation}\label{eqDiffop}
D=\sum_{|\alpha|\leqslant k}f_\alpha \, \partial_{x^1}^{\alpha_1}\cdots\partial_{x^{p+q}}^{\alpha_{p+q}} \partial_{\theta^1}^{\alpha_{p+q+1}} \cdots \partial_{\theta^{2r}}^{\alpha_{p+q+2r}} ,
\end{equation}
where $\partial_{y^i}$ stands for the partial derivative $\frac{\partial}{\partial y^i}$,  $\alpha$ is a multi-index, each $f_\alpha$ is in $\F$, $|\alpha|=\sum_{i=1}^{p+q+2r}\alpha_i$ and $\alpha_{p+q+1},\ldots,\alpha_{p+q+2r}$ are in $\{0,1\}$.
The natural action of $\Vect(\R^{p+q|2r})$ on $\mathcal{D}_{\lambda,\mu}$ is given by the supercommutator : for every $D\in\mathcal{D}_{\lambda,\mu}$ and $X\in\Vect(\R^{p|q})$,
\begin{equation}\label{formula:Lie derivative of DO}
\mathcal{L}_X D={\rm L}_X^\mu\circ D -(-1)^{\tilde{X}\tilde{D}}D\circ {\rm L}_X^\lambda .
\end{equation}
 
The $\Vect(\R^{p+q|2r})$-module of \emph{symbols} is the graded space associated with $\mathcal{D}_{\lambda,\mu}$. 
It is isomorphic to the module of weighted symmetric tensor fields 
\[
\mathcal{S}_{\delta} = \bigoplus_{k=0}^{\infty}\mathcal{S}^k_{\delta},\quad \delta=\mu-\lambda .
\]
The isomorphism comes from the \emph{principal symbol operator}
\[
\sigma_k : \mathcal{D}^k_{\lambda,\mu}\to \mathcal{S}^k_{\delta}: D\mapsto \sum_{|\alpha| = k}f_\alpha\otimes e_1^{\alpha_1}\vee\cdots \vee e_{p+q}^{\alpha_{p+q}}\vee e_{p+q+1}^{\alpha_{p+q+1}}\vee\cdots\vee e_{p+q+2r}^{\alpha_{p+q+2r}} ,
\]
if $D$ reads as (\ref{eqDiffop}). 
This operator commutes with the action of vector fields and induces a bijection from the quotient space $\mathcal{D}^k_{\lambda,\mu}/\mathcal{D}^{k-1}_{\lambda,\mu}$ to $\mathcal{S}^k_{\delta}$.

\subsection{Orthosymplectic algebras}\label{osp}
We recall the definition and the decomposition of the Lie superalgebra $\mathfrak{osp}(p+1,q+1|2r)$.
It will be convenient to relabel the elements of $\mathcal{B}_{p+q+2|2r}$ as $(e_{\mathfrak{o}},e_1,\ldots,e_{p+q},e_{\mathfrak{o}'},e_{p+q+1},\ldots,e_{p+q+2r})$ in order to particularize the first and last basis elements of the even subspace. We relabel the elements of $\mathcal{B}^*_{p+q+2|2r}$ accordingly and we identify $\R^{p+q|2r}$ to the subspace of $\R^{p+q+2|2r}$ made of elements whose components along $e_{\mathfrak{o}}$ and $e_{\mathfrak{o}'}$ vanish. With our notation, this linear embedding $\iota\colon\R^{p+q|2r}\to\R^{p+q+2|2r}$ maps $e_i\in \mathcal{B}_{p+q|2r}$ to $e_i\in \mathcal{B}_{p+q+2|2r}$, for $i\in I_{p+q+2r}$.

The orthosymplectic algebra $\mathfrak{osp}(p+1,q+1|2r)$ is the Lie subalgebra of $\mathfrak{gl}(p+q+2|2r)$ made of those matrices $A$ that preserve a particular supersymmetric even bilinear form $\omega$, in the sense that 
\begin{equation}\label{formula:preserve omega}
\omega(AU,V)+(-1)^{\tilde{A}\tilde{U}}\omega(U,AV)=0 , \mbox{ for all }U,V \in \R^{p+q+2|2r} . 
\end{equation}
This particular form $\omega$ is defined on $\R^{p+q+2|2r}$ by $\omega(U,V)=V^tGU$, where
\[
G=\left(\begin{array}{cc}S&0\\0&J\end{array}\right),\quad S=\left(\begin{array}{ccc}0&0&-1\\0&\Id_{p,q}&0\\-1&0&0\end{array}\right),\quad J=\left(\begin{array}{cc}0&\Id_r\\-\Id_r&0\end{array}\right) ,
\]
and $\Id_{p,q}$ takes the signature of the even part into account:
\[
\Id_{p,q}=\left(\begin{array}{cc}\Id_p&0\\0&-\Id_q\end{array}\right) .
\]
The musical isomorphisms associated with $\omega$ are defined as ususal by
\[
\flat : \R^{p+q+2|2r}\to (\R^{p+q+2|2r})^*: v\mapsto v^\flat= \omega(v,\cdot), \quad \sharp=\flat^{-1} .
\] 
Using these isomorphisms, we can easily see that the operators 
\begin{equation}\label{oij}
O_{i}^j=e_i\otimes \varepsilon^j-(-1)^{\tilde{i}\tilde{j}}(\varepsilon^j)^\sharp\otimes e_i^\flat ,
\end{equation}
where $i,j\in\{\mathfrak{o},\mathfrak{o}'\}\cup I_{p+q+2r}$, generate the orthosymplectic algebra\footnote{Actually, the operators $O_i^j$ for $i,j\in \{\mathfrak{o}'\}\cup I_{p+q+2r}$ are enough to generate the algebra (see Relations (\ref{rel})).}. 

In the purely even context, the algebra $\mathfrak{so}(p+1,q+1)$ has a natural decomposition into a direct sum of subalgebras (\cite{DuvLecOvs99}). 
This decomposition can be extended to the super context as follows.  The pull-back of the form $\omega$ by the embedding $\iota$ defines an even supersymmetric form $\omega_0$ on $\R^{p+q|2r}$, and the operators that preserve $\omega_0$ form the algebra $\mathfrak{osp}(p,q|2r)$.
Now, any matrix $A$ of $\mathfrak{osp}(p+1,q+1|2r)$ can be written as
\[
A=\left(\begin{array}{ccc|c} -a_1&v_1^t\Id_{p,q}&0&-v_2^tJ\\\Id_{p,q}\xi_1^t&B_1&v_1&B_2\\0&\xi_1&a_1&\xi_2\\\hline-J\xi_2^t&B_3&v_2&B_4\end{array}\right) ,
\]
where $a_1\in\R$, $v_1\in\R^{p+q}$, $v_2\in\R^{2r}$, $\xi_1\in(\R^{p+q})^*$ and $\xi_2\in(\R^{2r})^*$.
This decomposition of matrices defines a linear bijection 
\[
\Phi:\mathfrak{osp}(p+1,q+1|2r) \to \g_{-1}\oplus\g_0\oplus\g_1: A\mapsto (v,B-a_1\Id,\xi) ,
\] 
where $v=\left(\begin{array}{c}v_1\\v_2\end{array}\right)\in\g_{-1}=\R^{p+q|2r}$, $\xi=(\xi_1,\xi_2)\in\g_1=(\R^{p+q|2r})^*$, and 
\[
B=\left(\begin{array}{cc}B_1&B_2\\B_3&B_4\end{array}\right)\in\h_0=\mathfrak{osp}(p,q|2r) .
\]
Note that $\g_0=\h_0\oplus\R\euler$, where the element $\euler=-\Id$ corresponds to $a_1=1$ above and characterizes the decomposition of the algebra:
\[
\mathrm{ad}(\euler)|_{\g_i}=i\Id_{\g_i}\quad\mbox{for all }i\in\{-1,0,1\} .
\] 

We can compute the bracket in the algebra $\g_{-1}\oplus\g_0\oplus\g_1$: $\g_{-1}$ and $\g_1$ are commutative Lie superalgebras, the adjoint action of $\g_0$ on $\g_{-1}$ and $\g_1$ is given by the natural action of $\mathfrak{osp}(p,q|2r)\oplus\R\Id$ on $\R^{p+q|2r}$ and $(\R^{p+q|2r})^*$ respectively. 
Finally, for any $v\in\g_{-1}$ and $\xi\in\g_1$, 
\begin{equation}\label{bracketvxi}
[v,\xi]=v\otimes \xi- (-1)^{\tilde{v}\tilde{\xi}}\xi^\sharp\otimes v^\flat+(-1)^{\tilde{v}\tilde{\xi}}\langle \xi, v\rangle \Id,
\end{equation}
where $\langle \xi, v\rangle$ denotes the standard matrix multiplication of the row $\xi$ by the column $v$.
It will be convenient for our purpose to express the isomorphism $\Phi$ in terms of generators: it is easy to see that $\Phi$ actually maps $O_{i}^{j}$ to $O_{i}^{j}$, $O_{i}^{\mathfrak{o}'}$ to $e_i$, $O_{\mathfrak{o}'}^{i}$ to $\varepsilon^i$ and $O_{\mathfrak{o}'}^{\mathfrak{o}'}$ to $-\Id$, if $i,j\in I_{p+q+2r}$.

\subsection{The Lie superalgebra of vector fields $\mathfrak{osp}(p+1,q+1|2 r)$}\label{ospvect}
Let us recall how it is possible to realize the Lie superalgebra $\mathfrak{osp}(p+1,q+1|2r)$ as a Lie superalgebra 
of vector fields over the superspace $\R^{p+q|2 r}$. 
The construction is a superization of the classical construction of the algebras of conformal vector fields in the purely even situation, which we first recall.

The isotropic cone of the metric $S$ is the zero locus of the function
\[
F(x^{\mathfrak{o}},x^1,...,x^{p+q},x^{\mathfrak{o}'}) = (x^1)^2 + \cdots + (x^p)^2 - (x^{p+1})^2 - \cdots - (x^{p+q})^2 - 2x^{\mathfrak{o}} x^{\mathfrak{o}'} .
\]
The group $O(p+1,q+1)$ is made of those matrices that preserve $S$. 
Its linear action on $\R^{p+q+2}$ restricts to the isotropic cone and induces an action on the projective quadric associated with $F$. 
The space $\R^{p+q}$, viewed as a chart of this quadric, inherits a local action of $O(p+1,q+1)$. 
Differentiating this local action, we realize $\mathfrak{so}(p+1,q+1)$ as a Lie algebra of vector fields on $\R^{p+q}$. 
Equivalently, the above construction can be presented in terms of functions.  
Any function $f$ on $\R^{p+q}$ can be lifted to a homogeneous function $\hat{f}$ on $\R^{p+q+2} \setminus \{ x^{\mathfrak{o}'}=0 \}$ given by
\[
\hat{f}(x^{\mathfrak{o}},x^1,\ldots,x^{p+q},x^{\mathfrak{o}'}) = f\left(\frac{x^1}{x^{\mathfrak{o}'}},\ldots,\frac{x^{p+q}}{x^{\mathfrak{o}'}}\right) .
\]
We denote by $\chi^*(f)$ the restriction of $\hat{f}$ to the corresponding open subset of the isotropic cone of $S$ (i.e. $\hat{f}$ modulo the ideal generated by $F$).
Now we embed the Lie algebra $\mathfrak{so}(p+1,q+1)$ into $\Vect(\R^{p+q})$ by restricting the standard homomorphism
\[
h_{p+q+2}:\mathfrak{gl}(p+q+2)\to \Vect(\R^{p+q+2}) : A \mapsto Y^A = -\sum_{i,j} A_j^iy^j\partial_{y^i} 
\]
to $\mathfrak{so}(p+1,q+1)$ and by using $\chi^*$ to define $X^A = (\chi^*)^{-1} \circ Y^A \circ \chi^*$. 

In the super setting, the ingredients of the construction can be generalized as follows. First, any superfunction $f\in C^{\infty}(\R^{p+q|2r})$ has a decomposition
\[
f=\sum_{I\subseteq \{1,\ldots,2r\}}f_I(x^1,\ldots,x^{p+q})\theta^I ,
\]
where $f_I$ are smooth functions on $\R^{p+q}$ and $\theta^{I} = \theta^{i_1}\cdots\theta^{i_a}$ if $i_1<\cdots<i_a$ are the elements of $I$. We define 
\[
\hat{f}(x^{\mathfrak{o}},x^1,\ldots,x^{p+q},x^{\mathfrak{o}'},\theta^1,\ldots,\theta^{2r}) = 
\sum_{I\subseteq \{1,\ldots,2r\}} f_I\left(\frac{x^1}{x^{\mathfrak{o}'}},\ldots,\frac{x^{p+q}}{x^{\mathfrak{o}'}}\right) \frac{\theta^I}{(x^{\mathfrak{o}'})^{|I|}} .
\]
The superfunction $\hat{f}$ is homogeneous in the sense that each $\hat{f_I}$ is homogeneous of degree $-|I|$. We denote by $\chi^*(f)$ the restriction of $\hat{f}$ to the supercone of $G$ whose equation is $F(x,\theta)=0$, where
\[F(x,\theta)=\sum_{i=1}^p(x^i)^2-\sum_{i=p+1}^{p+q}(x^i)^2-2x^{\mathfrak{o}}x^{\mathfrak{o}'}+2\sum_{i=1}^r\theta^i\theta^{i+r},\]
i.e. we consider $\hat{f}$ modulo the ideal generated by $F$.
Then the superalgebra $\gl(p+q+2|2r)$ can be realized as a subalgebra of vector fields of $\R^{p+q+2|2r}$ by means of the homomorphism
\[
h_{p+q+2|2r}:\mathfrak{gl}(p+q+2|2r)\to \Vect(\R^{p+q+2|2r}) : 
A \mapsto -\sum_{i,j}(-1)^{\tilde{j}(\tilde{i}+\tilde{j})}A_j^iy^j\partial_{y^i} .
\]
Doing as above, we can associate with each element $h$ of $\Phi(\mathfrak{osp}(p+1,q+1|2r))$ a vector field on $\R^{p+q+2}$, namely
\begin{equation}\label{real}
X^h=\left\{\begin{array}{ll}
-\sum_{i=1}^{p+q+2r}h^i\partial_{y^i}&\mbox{ if }h\in\R^{p+q|2r} = \g_{-1} ,\\
-\sum_{i,j=1}^{p+q+2r}(-1)^{\tilde{j}(\tilde{i}+\tilde{j})}h_j^i
y^j\partial_{y^i}&\mbox{ if }h\in \g_0 ,\\
\sum_{j=1}^{p+q+2r}h_j y^j(-1)^{\tilde{j}}X^\euler + \frac{1}{2}F_0(y)X^{h^\sharp}&\mbox{ if }h\in(\R^{p+q|2r})^*  = \g_{1},
\end{array}
\right.
\end{equation}
where 
\[F_0(y)=\sum_{i=1}^p(y^i)^2-\sum_{i=p+1}^{p+q}(y^i)^2+2\sum_{i=p+q+1}^{p+q+r}y^iy^{i+r}.\]
Note that for every $A\in\g_0 = \mathfrak{osp}(p,q|2r) \oplus \R \Id$, Formula (\ref{eqRadouxLeites}) reduces to 
\begin{equation}\label{eqXA}
\mathrm{L}_{X^A}(f\otimes v)=X^A(f)\otimes v+(-1)^{\tilde{A}\tilde{f}}f\otimes \rho(A)v .
\end{equation}

\subsection{Equivariant quantizations}
By a \emph{quantization} on $\R^{p+q|2r}$, we mean a linear bijection $Q$ from the space of symbols $\mathcal{S}_{\delta}$ to the space of differential operators $\mathcal{D}_{\lambda,\mu}$ that preserves the principal symbol, i.e.,
\begin{equation}\label{norma}
\sigma_k(Q(T))=T
\end{equation}
for all $k\in\mathbb{N}$ and all $T \in \mathcal{S}^{k}_{\delta}$.  
The inverse map of such a quantization is called \emph{a symbol map}.
A quantization is \emph{$\mathfrak{osp}(p+1,q+1|2r)$-equivariant} if 
\[
\mathcal{L}_{X^h}\circ Q=Q\circ \mathrm{L}_{X^h}
\]
for all $h\in \mathfrak{osp}(p+1,q+1|2r)$.

\section{Tools for the quantization}\label{Tools}

Here we adapt in the orthosymplectic setting the basic tools that were used to build the quantization in the
purely even situation or in the super projective case \cite{MatRad11}.

\subsection{The affine quantization and the map $\gamma$}
We first introduce the affine quantization map and we use this map to carry the $\Vect(\R^{p+q|2r})$-module structure of $\mathcal{D}_{\lambda,\mu}$ to $\S_\delta$. Then we focus on the map $\gamma$ which measures the difference between the obtained $\Vect(\R^{p+q|2r})$-module structure on $\S_\delta$ and the one given by the Lie derivative of symbols.

The \emph{affine quantization map} $Q_{\mathrm{Aff}}$ is defined as the inverse of the total symbol map 
\[
\sigma_{\mathrm{Aff}} : \mathcal{D}_{\lambda,\mu}\to \mathcal{S}_{\delta}: D \mapsto \sum_{|\alpha|\leqslant k}f_\alpha\otimes e_1^{\alpha_1}\vee\cdots \vee e_p^{\alpha_p}\vee e_{p+1}^{\alpha_{p+1}}\vee\cdots\vee e_{p+q}^{\alpha_{p+q}} ,
\]
when $D$ is given by (\ref{eqDiffop}). 
The map $Q_{\mathrm{Aff}}$ is an equivariant quantization with respect to the superalgebra made of constant and linear super vector fields. Moreover, $Q_{\mathrm{Aff}}$ can be expressed in a coordinate-free manner.
\begin{proposition}
For every $v_1,\ldots, v_k\in\R^{p+q|2r}$, $f\in \mathrm{C}^{\infty}(\R^{p+q|2r})$, we have 
\[
Q_{\mathrm{Aff}}(f\otimes v_1\vee\cdots\vee v_k) = (-1)^k f\;\mathrm{L}_{X^{v_1}}\circ\cdots\circ \mathrm{L}_{X^{v_k}} .
\]
\end{proposition}
We carry the $\Vect(\R^{p+q|2r})$-module structure of $\mathcal{D}_{\lambda,\mu}$ to $\S_\delta$ by defining
\[
\mathcal{L}_XT=Q_{\mathrm{Aff}}^{-1}\circ\mathcal{L}_X\circ Q_{\mathrm{Aff}}(T)
\]
for all $T$ in $\mathcal{S}_{\delta}$ and all $X$ in $\Vect(\R^{p+q|2r})$.
The difference between the representations $(\mathcal{S}_{\delta},\L)$ and $(\mathcal{S}_{\delta},\mathrm{L})$ is measured by the map 
\[
\gamma : \g\to \gl(\S_{\delta},\S_{\delta}) : h\mapsto \gamma(h)=\L_{X^h}-\mathrm{L}_{X^h} .
\]
An easy computation in coordinates yields its basic properties as in \cite{BonMat06,MatRad11}.
\begin{proposition}\label{gamma0}                                                                                                                                                                        
The map $\gamma$ vanishes on $\g_{-1}\oplus \g_0$. 
Moreover, for every $h\in\g_1$ and $k\in \N$, $\gamma(h)$ maps $\S^k_\delta$ to $\S^{k-1}_\delta$ and is a differential operator with constant coefficients, of order zero and parity $\tilde{h}$.
 \end{proposition}

We now derive a coordinate-free expression of $\gamma$. 
To this aim, we recall that the interior product of a row vector $\varepsilon\in(\R^{p+q|2r})^*$ in a symmetric tensor $v_1\vee\cdots\vee v_k\in S^k$, namely 
\begin{equation}\label{interior1}
\mathrm{i}(\varepsilon)(v_1\vee\cdots\vee v_k)=\sum_{a=1}^k(-1)^{\tilde{\varepsilon}(\sum_{b=1}^{a-1}\tilde{v_b})}\langle \varepsilon,v_a\rangle v_1\vee\cdots \widehat{a}\cdots\vee v_k ,
\end{equation}
can be extended to $\mathcal{S}_\delta$ by setting $\mathrm{i}(\varepsilon)u=0$ for $u$ in $B^\delta$ and by defining $\mathrm{i}(\varepsilon)$ as a differential operator of order zero and parity $\tilde{\varepsilon}$.
This interior product extends to symmetric covariant two-tensors by setting
\begin{equation}\label{interior2}
\mathrm{i}(\varepsilon\vee \varepsilon')S=\mathrm{i}(\varepsilon)\circ \mathrm{i}(\varepsilon')S
\end{equation}
for all $\varepsilon,\varepsilon'\in (\R^{p+q|2r})^*$ and $S\in \S_\delta$.  
Then we introduce the operators 
\begin{equation}\label{eq:T}
T\colon\S^k_\delta\to\S^{k-2}_\delta\colon S\mapsto\sum_{j=1}^{p+q+2r}\mathrm{i}(e_j^\flat\vee\varepsilon^j)S ,
\end{equation}
and extending in the same way the symmetric product by elements of $\R^{p+q|2r}$ to $\S^k_\delta$,
\begin{equation}\label{eq:R}
R\colon\S^k_\delta\to\S^{k+2}_\delta\colon S\mapsto \sum_{j=1}^{p+q+2r}e_j\vee(\varepsilon^j)^\sharp\vee S.
\end{equation}
An explicit form of the operator $T$ is given by the following result.
\begin{lemma}\label{innerprod}
 We have 
 \[T(v_1\vee\cdots\vee v_k)=2\sum_{b=1}^k\sum_{a<b}(-1)^{\tilde{v_a}(\sum_{a<c<b}\tilde{v_c})}\omega_0(v_a,v_b)v_1\vee\cdots \hat{a}\cdots\hat{b}\cdots\vee v_k\]
 for every homogeneous $v_1,\ldots,v_k\in\R^{p+q|2r}$.
\end{lemma}
\begin{proof}
Using (\ref{interior1}), (\ref{interior2}) and the definition of $T$, we get immediately that the left-hand side is equal to
\begin{multline*}
\sum_{j=1}^{p+q+2r}\left(\sum_{b=1}^k\sum_{a<b} (-1)^{\tilde{j}(\sum_{c=a}^{b-1}\tilde{v_c})}\langle e_j^\flat,v_a\rangle \langle\varepsilon^j,v_b\rangle\,v_1\vee\cdots\hat{a}\cdots\hat{b}\cdots\vee v_k\right.\\+
\left.\sum_{b=1}^k\sum_{a>b} (-1)^{\tilde{j}(\sum_{c=b}^{a-1}\tilde{v_c}+\tilde{v_b})}\langle e_j^\flat,v_a\rangle \langle\varepsilon^j,v_b\rangle\,v_1\vee\cdots\hat{b}\cdots\hat{a}\cdots\vee v_k\right).
\end{multline*}
The result follows by permuting the sums, exchanging the indexes $a$ and $b$ in the second summand, and using the definition of $\flat$. 
\end{proof}

The expression of $\gamma$ is detailed in the following result whose proof is similar to the corresponding one in \cite{MatRad11}.  
\begin{proposition}\label{gamma1}
For every $h\in \g_1 = (\R^{p+q|2r})^*$ we have on ${\mathcal S}^k_\delta$
\[
\gamma(h)=-(\lambda d+k-1)\mathrm{i}(h)+\frac{1}{2}h^\sharp\vee T ,
\]
where $\mathrm{i}(h)$ and $T$ are given by (\ref{interior1}) and (\ref{eq:T}), respectively. 
\end{proposition}
\begin{proof}
It is sufficient to show that both sides of the equality agree on tensors of the form $v_1\vee\cdots\vee v_k$ with homogeneous $v_i\in\R^{p+q|2r}$. 
By the definition of $\gamma$, we have 
\[
Q_{\mathrm{Aff}}(\gamma(h)(v_1\vee\cdots\vee v_k))=\mathcal{L}_{X^{h}}(Q_{\mathrm{Aff}}(v_{1}\vee\cdots\vee v_{k}))-
 Q_{\mathrm{Aff}}(\mathrm{L}_{X^{h}}(v_{1}\vee\cdots\vee v_{k})) .
 \]
By Proposition \ref{gamma0} and the definition $Q_{\mathrm{Aff}}$, the left-hand side is a differential operator of order $k-1$. 
Hence, we only have to sum up the terms of that order in the first term of the right-hand side, which is equal to
\[
(-1)^{k}[\mathrm{L}_{X^{h}}\circ \mathrm{L}_{X^{v_{1}}}\circ\cdots\circ \mathrm{L}_{X^{v_{k}}} -(-1)^{\tilde{h}(\tilde{v_{1}}+\cdots+\tilde{v_{k}})} \mathrm{L}_{X^{v_{1}}}\circ\cdots\circ \mathrm{L}_{X^{v_{k}}}\circ \mathrm{L}_{X^{h}}] .
\]
An easy computation shows that this expression  can be rewritten as
\begin{multline*}
(-1)^{k}[\sum_{i=1}^k(-1)^{\tilde{v_i}(\sum_{l=1}^{i-1}\tilde{v_{l}})} \mathrm{L}_{X^{[h,v_i]}}\circ
 \mathrm{L}_{X^{v_1}}\circ\cdots \widehat{i}\cdots \circ \mathrm{L}_{X^{v_k}}\\
 + \sum_{i=1}^k\sum_{j=1}^{i-1}(-1)^{\tilde{h}(\sum_{l=1}^{i-1}\tilde{v_{l}})+(\tilde{h}+\tilde{v_{i}})(\sum_{l=j+1}^{i-1}\tilde{v_{l}})} \mathrm{L}_{X^{v_1}}\circ\cdots\underbrace{\mathrm{L}_{X^{[v_j,[h,v_i]]}}}_{(j)}\cdots\widehat{i}\cdots\circ \mathrm{L}_{X^{v_k}}] .
\end{multline*}
It follows from Equations (\ref{eqRadouxLeites}) and (\ref{eqXA}) that the term of order $k-1$ in the first summand is exactly
\begin{multline*}
(-1)^{k}\sum_{i=1}^k(-\lambda\,\str([h,v_i]))(-1)^{\tilde{v_i}(\sum_{l=1}^{i-1}\tilde{v_{l}})}
\mathrm{L}_{X^{v_1}}\circ\cdots\widehat{i}\cdots \circ \mathrm{L}_{X^{v_k}}\\
=Q_{\mathrm{Aff}}(-\lambda d \, \mathrm{i}(h)(v_1\vee\cdots\vee v_k))
\end{multline*}
In order to deal with the second term, we compute the bracket
\[
[v_j,[h,v_i]]
=\langle h,v_{i}\rangle v_{j}+(-1)^{\tilde{h}\tilde{v_{j}}}\langle h,v_{j}\rangle v_{i}-(-1)^{\tilde{v_j}(\tilde{h}+\tilde{v_i})}h^\sharp\,\omega_0(v_i,v_j) .
\]
As in \cite{MatRad11}, the first two terms yield
\[
-(k-1)Q_{\mathrm{Aff}}(\mathrm{i}(h)(v_1\vee\cdots\vee v_k)) .
\]
Finally, the last term of the bracket yields 
\[
Q_{\mathrm{Aff}}(\sum_{i=1}^k\sum_{j<i}(-1)^{\tilde{h}(\sum_{l=1}^{i-1}\tilde{v_{l}})+(\tilde{h}+\tilde{v_{i}})(\sum_{l=j}^{i-1}\tilde{v_{l}})}\omega_0(v_i,v_j)v_1\vee\cdots\underbrace{h^\sharp}_j\stackrel{\hat{i}}{\cdots}\vee v_k) ,
\]
that is, using Lemma \ref{innerprod},
\[
Q_{\mathrm{Aff}}(\frac{1}{2}h^\sharp\vee T(v_1\vee\cdots\vee v_k)) ,
\]
and the result follows.
\end{proof}

\subsection{Casimir operators}
As in the purely even context \cite{DuvLecOvs99,BonMat06,BonHanMatPon02} or the super projective case \cite{MatRad11}, we will use quadratic Casimir operators (see \cite{Kac77,Ber87,Pin90,Mus97,Ser99,SerLei02,Sch83} for detailed descriptions) to build the quantization. 
We will show that there is a simple relation between the Casimir operators $C$ and $\mathcal{C}$ associated with the representations $({\mathcal S}_\delta,\mathrm{L})$ and $({\mathcal S}_\delta,\L)$ of $\g=\mathfrak{osp}(p+1,q+1|2r)$, respectively and we will compute an explicit form of $C$.

Given a representation $(E,\beta)$ of $\mathfrak{osp}(p+1,q+1|2r)$, we consider the second order Casimir operator $C_\beta$ defined by
\[
C_\beta = \sum_{i}\beta(u^*_{i})\beta(u_i)
\]
where $u_i$ and $u^*_i$ are bases of $\mathfrak{osp}(p+1,q+1|2r)$, dual with respect to the Killing form
\begin{equation}\label{Killing}
K\colon\g\times\g\to\R\colon(A,B)\mapsto-\frac{1}{2}\mathrm{str}(AB) ,
\end{equation}
in the sense that $K(u_i,u_j^*) = \delta_{i,j}$

\begin{lemma}\label{lemmacas}
The second order Casimir operator associated with a representation $(E,\beta)$ of $\mathfrak{osp}(p+1,q+1|2r)$ is 
\begin{multline}\label{eq:C1}
C_\beta=-2\sum_{i\leqslant p+q+2r}(-1)^{\tilde{i}}\beta(\varepsilon^i)\beta(e_i)+d\beta(\euler)\\
-\beta(\euler)^2-\frac{1}{2}\sum_{i,j\leqslant p+q+2r}(-1)^{\tilde{i}}\beta(O_{j}^{i})\beta(O_{i}^{j}).
\end{multline} 
\end{lemma}
\begin{proof}
We first show how to select generators $O_i^j\in\mathfrak{osp}(p+1,q+1|2r)$ to define a basis of this algebra. 
We define $\pi$ as the permutation of $I=I_{p+q+2r}\cup\{\mathfrak{o},\mathfrak{o}'\}$ that exchanges $\mathfrak{o}$ and $\mathfrak{o}'$, fixes the elements of $I_{p+q}$ and exchanges $i$ and $i+r$ for $i\in \{p+q+1,\ldots,p+q+r\}$, and the function $s$ by
\[
s(i)= 
\left\{\begin{array}{ll}
-1,&\mbox{if }i\in\{\mathfrak{o},\mathfrak{o}'\}\cup\{p+1,\ldots, p+q+r\} ,\\
1,&\mbox{otherwise.} 
\end{array}\right .
\]
It is then easy to see that the generators obey the relation
\begin{equation}\label{rel}
O_{\pi(j)}^{\pi(i)}=a_{i,j}O_{i}^{j} ,
\end{equation}
where 
\[
a_{i,j}=-(-1)^{\tilde{i}\tilde{j}}s(\pi(j))s(i) .
\] 
In particular, $O_{i}^{\pi(i)}=0$ for every even index $i$. 
Therefore, for every set $A\subset I\times I$ that contains exactly one element of each set $\{(i,j),(\pi(j),\pi(i))\}$ when $j\not=\pi(i)$ or $i$ is odd, the collection $\{O_{i}^j,(i,j)\in A\}$ defines a basis of $\mathfrak{osp}(p+1,q+1|2r)$. %
Observe that, if $A$ is such a set, then so is $A_\pi=\{(\pi(j),\pi(i)):(i,j)\in A\}$. 

Now, a straightforward computation shows that
\[
K(O_{i}^{j},O_{k}^{l})=-(-1)^{\tilde{i}}(\delta_{j,k}\delta_{i,l}+a_{\pi(i),\pi(j)}\delta_{j,\pi(l)}\delta_{i,\pi(k)}) .
\]
Therefore, defining for any $i,j\in I$
\begin{equation}\label{oijstar}
O_{i}^{j*}=-(-1)^{\tilde{i}}(1+\delta_{i,\pi(j)})^{-1}O_{j}^{i} ,
\end{equation}
we have $K(O_{i}^{j},O_{k}^{l*})=\delta_{i,k}\delta_{j,l}$ for every $(i,j)$ and $(k,l)$ in $A$. 
By definition, the quadratic Casimir operator associated with $(E,\beta)$ is thus equal to 
\[
C_\beta=\sum_{(i,j)\in A}\beta(O_{i}^{j*})\beta(O_{i}^{j})=\sum_{(i,j)\in A_\pi}\beta(O_{i}^{j*})\beta(O_{i}^{j}) .
\]
Summing both expressions of $C_\beta$, the Casimir operator can be computed as a sum over all the indices  
\[
C_\beta=-\frac{1}{2}\sum_{i,j\in I}(-1)^{\tilde{i}}\beta(O_{j}^{i})\beta(O_{i}^{j}) .
\]
We gather the terms containing the indices $\mathfrak{o}$ and $\mathfrak{o}'$, take into account (\ref{rel}) and the isomorphism $\Phi$, use the relation
\[
\sum_{i\leqslant p+q+2r}[e_i,\varepsilon^i]=-d\euler ,
\]
and the result follows. 
\end{proof}

We define the operator
\[
N\colon \mathcal{S}^k_\delta\to\mathcal{S}^{k-1}_\delta \colon S\mapsto -2\sum_{i=1}^{p+q+2r}(-1)^{\tilde{i}}\gamma(\varepsilon^i)\mathrm{L}_{X^{e_i}}S .
\]
We apply Lemma \ref{lemmacas} for both $C$ and $\cc$, and use Proposition \ref{gamma0} to obtain the relation between these operators.
\begin{proposition}\label{relcas}
The Casimir operators are related by
\begin{equation}\label{casinil}
\cc=C+N .
\end{equation}
\end{proposition}
We now state the main result about the operator $C$, using the operators $R$ and $T$ (see (\ref{eq:R}) and (\ref{eq:T})). 
\begin{proposition}\label{Cassl}
The Casimir operator $C$ associated with the representation $(\S^k_\delta,\mathrm{L})$ of $\mathfrak{osp}(p+1,q+1|2r)$ is given by
\[
C=-[(-k+d\delta)^2-d(-2k+d\delta)+ k^2-2k]\mathrm{Id}+ R\circ T .
\]
\end{proposition}
\begin{proof}
Since the Casimir operator $C$ commutes with the action of constant vector fields on $\S^k_\delta$, it has constant coefficients. 
Therefore we just collect the terms with constant coefficients in (\ref{eq:C1}), with $\beta=\mathrm{L}_X$. 
Using (\ref{eqXA}), we see that these terms read 
\[
-[(-k+d\delta)^2-d(-k+d\delta)]\mathrm{Id}-\frac{1}{2}\sum_{i,j\leqslant p+q+2r}(-1)^{\tilde{i}}\rho(O_{j}^{i})\rho(O_{i}^{j}) ,
\]
where $\rho$ is given by (\ref{rhotens}) and the result follows from a straightforward computation of the last summand.
\end{proof}

\subsection{Spectrum of $C$}
In the purely even context, it was shown in \cite{DuvLecOvs99} that the Casimir operator $C$ is diagonalizable. 
Indeed, in view of Proposition \ref{Cassl}, this amounts to show that the operator $R\circ T$ is diagonalizable. 
But the eigenspace decomposition of this operator is given by the decomposition of the space of symmetric tensors into spherical harmonics. 
In the super setting, this decomposition still exists and $C$ is diagonalizable provided the superdimension does not belong to $-2\N$.
We compute in general the minimal polynomial of $C$ to obtain its spectrum also in the latter special situation.
\begin{definition}
For any $k, s \in {\mathbb N}$, we set 
\begin{equation}\label{formula:b_ks}
b_{k,s} = 2s(d+2(k-s-1)).
\end{equation}
\end{definition}
\begin{lemma}\label{lem:(RT)^r}
For any $k \in {\mathbb N}$ and $s\in\{0,\ldots, \lfloor \frac{k}{2}\rfloor\}$, we have   
\begin{equation}\label{formula:(RT)^r}
R^{s+1}\circ T^{s+1}= \prod_{i=0}^{s}(R\circ T-b_{k,i}),
\end{equation}
on ${\mathcal S}_{\delta}^k$. 
\end{lemma}
\begin{proof}
The result holds for $s=0$ since $b_{k,0}=0$. 
Now we have
\[
R^{s+1}\circ T^{s+1}=R\circ [R^s,T]\circ T^s+R\circ T\circ R^s\circ T^s .
\]
We can compute that $[R^s, T]|_{\S_{\delta}^{k-2s}}=-b_{k,s}R^{s-1}|_{\S_\delta^{k-2s}}$ and the result follows by induction.
\end{proof}
\begin{definition}
For any $k, s \in {\mathbb N}$, we set 
\begin{equation}\label{formula:alpha_ks}
\alpha_{k,s,\delta} = -[(-k+d\delta)^2-d(-2k+d\delta)+ k^2-2k]\mathrm{Id}+b_{k,s}.
\end{equation}
\end{definition}
\begin{proposition}\label{minpoly}
The minimal polynomial of $C|_{\S^k_\delta}$ is
\[
\prod_{s=0}^{\lfloor k/2 \rfloor}\left(x-\alpha_{k,s,\delta} \right) .
\]
In particular, the spectrum of $C|_{\S^k_\delta}$ is $\{\alpha_{k,s,\delta}\colon 0\leqslant s \leqslant \lfloor k/2 \rfloor\}$.
\end{proposition}
\begin{proof}
By Proposition \ref{Cassl} and Lemma \ref{lem:(RT)^r}, we see that the considered polynomial annihilates the operator $C|_{\S^k_\delta}$. 
If it was not minimal the operators $(R\circ T)^l$, for $0 \leqslant l \leqslant \lfloor k/2 \rfloor$ would be linearly dependent and so would be the operators $R^l\circ T^l$ (still by Lemma \ref{lem:(RT)^r}). 
But this is not possible since we have clearly $\ker T^{l-1} \setminus \ker T^{l}\not=\emptyset$ for $0 \leqslant l \leqslant \lfloor k/2 \rfloor$.
\end{proof}
\begin{remark}\label{remark:2.12}
We can compute easily that two roots $\alpha_{k,s,\delta}$ and $\alpha_{k,s',\delta}$ of the minimal polynomial coincide when $s=s'$ or 
\begin{equation}\label{ss}
2(s+s')=d+2k-2 .
\end{equation}
Since $s$ and $s'$ are in $\{0,\ldots,\lfloor\frac{k}{2}\rfloor\}$, (\ref{ss}) has solutions only  when $d$ is less than or equal to 0 and even, the simplest example being $d=0$, $k=2$, $s=0$ and $s'=1$. 
In this case, the Casimir operator $C|_{\S^k_\delta}$ is not diagonalizable. 
However, it also follows from (\ref{ss}) that the multiplicity of the roots of the minimal polynomial of $C|_{\S^k_\delta}$ is at most two.
\end{remark}

\section{Construction of the quantization}

In \cite{DuvLecOvs99}, the harmonic decomposition of symmetric tensors was used to prove that the Casimir operator $C$ of $\mathfrak{so}(p+1,q+1)$ is diagonalizable, and the quantization was built by associating an eigenvector of $\mathcal{C}$ to every homogeneous eigenvector of $C$. 
We showed in the previous section that the Casimir operator $C$ of $\mathfrak{osp}(p+1,q+1|2r)$ is not diagonalizable when the superdimension belongs to $-2\N$ because the harmonic decomposition does not exist in such cases. 
Here we will naturally adapt the method of \cite{DuvLecOvs99} by defining the quantization on \emph{generalized eigenvectors} of $C$ and by replacing the eigenvector equation \cite[Eq. (5.7)]{DuvLecOvs99} by a generalized eigenvector equation.
 
Before stating the main result, we define resonant values of the parameter $\delta$ as in \cite{DuvLecOvs99}. 
\begin{definition}
A value of $\delta$ is \emph{resonant} if there exist $k,l,i,j\in\N$ with $l<k$ such that $\alpha_{k,i,\delta}=\alpha_{l,j,\delta}$. 
\end{definition}
The following result shows that when the superdimension is not zero, the resonant values are nothing but the ones in the classical situation (see \cite[p.~2009]{DuvLecOvs99}) up to replacement of the dimension $n$ by the superdimension $d$. 
\begin{proposition}
If $d\not=0$, then the set of resonant values of $\delta$ is 
\[
\mathcal{R}=\{\delta_{k,l,s,t}\colon k,l,s,t\in\N,k>l,2s\leqslant k,2t\leqslant l\} ,
\]
where 
\[
\delta_{k,l,s,t}=\frac{k+l+d-1+s+t}{d}+\frac{(t-s)(d-2-2(t+s)+k+l)}{d(k-l)} .
\]
\end{proposition}
If $d=0$, any value of $\delta$ is resonant since $\alpha_{1,0,\delta}=\alpha_{0,0,\delta}$. Moreover, it turns out that the whole equivariant quantization problem does not depend on $\delta$.
We will discuss this special case in \ref{subsection:d=0} and assume from now on that $d\not=0$.
\begin{theorem}\label{flatex}
If $\delta$ is not resonant, there exists a unique $\mathfrak{osp}(p+1,q+1|2r)$-equivariant quantization from $\mathcal{S}_{\delta}$ to $\mathcal{D}_{\lambda,\mu}$.
\end{theorem}
\begin{proof}
First we observe that for any $k\in \N$, the space $\mathcal{S}_{\delta}^{k}$ is decomposed into a direct sum of generalized eigenspaces of $C$, that is
\[
\mathcal{S}_{\delta}^{k}=\oplus_{i}\ker(C-\alpha_{k,i,\delta}\Id)^{2} ,
\]
 (see Remark \ref{remark:2.12}). Then for every $S\in\mathcal{S}_{\delta}^{k}\cap\ker(C-\alpha_{k,i,\delta}\Id)^{2}$, there exists a unique $\hat{S}=S_k + S_{k-1}+\cdots+S_0\in\ker(\mathcal{C}-\alpha_{k,i,\delta}\Id)^{2}$ such that $S_k=S$ and $S_l\in\mathcal{S}_{\delta}^{l}$ for all $l\leqslant k-1$. Indeed, these conditions read $S_k=S$ and
\begin{equation}\label{eq2}
(C-\alpha_{k,i,\delta}\Id)^{2}S_{k-l}=-(C\circ N+N\circ C-2\alpha_{k,i,\delta}N)S_{k-l+1}-N^2S_{k-l+2},
\end{equation}
where $S_{k+1}=0$ and $S_l\in\mathcal{S}_{\delta}^{l}$ for all $l\in\{1,\ldots,k\}$.
As $\delta$ is not resonant, the operators $(C-\alpha_{k,i,\delta}\Id)\vert_{\mathcal{S}_{\delta}^{k-l}}$ are all invertible and therefore this system of equations has a unique solution. 

Now, define the quantization $Q$ by
\[
Q\vert_{\ker(C-\alpha_ {k,i,\delta}\Id)^{2}}(S)=\hat{S} .
\]
It is clearly a bijection and it also fulfills 
\[
Q\circ \mathrm{L}_{X^h} = \L_{X^h}\circ Q\quad\mbox{for all } h\in\,\mathfrak{osp}(p+1,q+1|2r) .
\]
Indeed, for any $S\in\mathcal{S}_\delta^k \cap \ker(C-\alpha_{k,i,\delta}\Id)^{2}$, the tensors $Q(\mathrm{L}_{X^h}S)$ and $ \L_{X^h}(Q(S))$
share the following properties:
\begin{itemize}
\item 
they belong to the space $\ker(\mathcal{C}-\alpha_{k,i,\delta}\Id)^{2}$ because, on the one hand, $\cc$ commutes with $\L_{X^h}$ for all $h$ and, on the other hand, $C$ commutes with $\mathrm{L}_{X^{h}}$ for all $h$.
\item 
their term of degree $k$ is exactly $\mathrm{L}_{X^h}S$.
\end{itemize}
The first part of the proof shows that they have to coincide.
\end{proof}

\begin{remark}
 Equation (\ref{eq2}) allows us to compute the quantization of generalized eigenvectors of $C$. It is easy to see that if we start with an eigenvector  $S\in\mathcal{S}_{\delta}^{k}\cap\ker(C-\alpha_{k,i,\delta}\Id)$, this equation reduces to the usual one, namely
\begin{equation}\label{eq1}
 (C-\alpha_{k,i,\delta}\Id)S_{k-l}=-NS_{k-l+1}.
\end{equation}
\end{remark}

\section{Explicit formulae}

Let us now provide explicit formulae for the quantization for symbols of degree at most two. 
We first consider the generic case $d\not=0$ and then we deal with the special case $d=0$.

\subsection{Nonzero superdimension}
We begin with the quantization of symbols of degree one. 
Remember that the divergence operator can be extended to $\S^k_\delta$ for every $k$ and $\delta$ by
\begin{equation}\label{divergence}
\dive \colon \S^k_{\delta} \to \S^{k-1}_{\delta} \colon S\mapsto \sum_{j=1}^{p+q+2r}(-1)^{\tilde{j}}\mathrm{i}(\varepsilon^{j})\partial_{y^j}S .
\end{equation}
\begin{proposition}
 When $d\not=0$ and $\delta\not=1$, the $\mathfrak{osp}(p+1,q+1|2r)$-equivariant quantization on $\mathcal{S}^1_\delta$ is given by
 \[
 Q=Q_\aff\circ(\Id+\frac{\lambda}{1-\delta}\div) .
 \]
 When $\delta=1$, the quantization does not exist unless $\lambda=0$.
\end{proposition}
\begin{proof}
 We just solve Equation (\ref{eq1}) and compute that $N|_{\mathcal{S}^1_\delta}=-2\lambda d \, \div$.
\end{proof}
For symbols of degree two, we will need to use the superization of well-known operators.

\begin{definition}
The gradient operator and the Lapacian are defined by the following formulae:
\begin{equation}\label{gradient}
\G \colon \S^k_{\delta} \to \S^{k+1}_{\delta} \colon S\mapsto \sum_{j=1}^{p+q+2r}(-1)^{\tilde{j}}\varepsilon^{j\sharp}\vee\partial_{y^j}S ,
\end{equation}
\begin{equation}\label{laplacian}
\Delta \colon \S^k_{\delta} \to \S^{k}_{\delta} \colon S\mapsto \sum_{j=1}^{p+q+2r}\omega_{0}(e_i,e_j)\partial_{y^j}\partial_{y^i}S .
\end{equation}
Moreover, we set $\G_{0}=\G\circ T$ and  $\Delta_0=\Delta\circ T$. 
\end{definition}

The result is then the superization of the classical one in \cite{DuvLecOvs99}.
\begin{proposition}
 If $d\not=0$ and $\delta$ is not resonant, then the $\mathfrak{osp}(p+1,q+1|2r)$-equivariant quantization on $\mathcal{S}^2_\delta$ is 
\[
Q=Q_{\mathrm{Aff}}\circ (\Id+a_{1}G_{0}+a_{2}\dive+a_{3}\Delta_{0}+a_{4}\dive^{2}) ,
\]
where the coefficients are given by
\begin{eqnarray*}
a_{1}&=&\frac {d(2\lambda+\delta-1)}{2(d\delta-2)(d(\delta-1)-2)} ,\\
a_{2}&=&\frac {-(\lambda d+1)}{d(\delta-1)-2} ,\\
a_{3}&=&\frac {d\lambda(2+(4\lambda-1)d+(-\delta^2 - 3\lambda\delta + 2 \lambda + 2 \delta -1)d^2)}{2(d(\delta-1)-1)(d(2\delta-1)-2)(d\delta-2)(d(\delta-1)-2)} ,\\
a_{4}&=&\frac {d\lambda(d\lambda+1)}{2(d(\delta-1)-1)(d(\delta-1)-2)}.
\end{eqnarray*}
\end{proposition}
\begin{proof}
First we write any symbol $S\in\mathcal{S}^2_\delta$ as a sum of eigenvectors of $C$:
\[
S=S_{2,1}+S_{2,0},\quad\mbox{where}\quad S_{2,1}=\frac{1}{2d}R\circ T(S).
\]
Solving Equation (\ref{eq1}), we obtain that the quantization is given by
\begin{multline*}
Q(S)=Q_\aff(S+\frac{1}{2 (d \delta-2)}N S_{2,1}+\frac{1}{4(d \delta-2)(2 d \delta-d-2)}N^2 S_{2,1}+\\
 +\frac{1}{2( d \delta-d-2)}N S_{2,0}+\frac{1}{8(d\delta-d-2)(d \delta-d-1)}N^2S_{2,0}).
\end{multline*}
We then conclude easily by using straightforward relations that hold on $\mathcal{S}^2_\delta$: $N=-2(\lambda d+1)\div+G_0$, $\div\circ R\circ T=2 G_0$, $G_0\circ R\circ T=2d G_0$, and finally $\div\circ G=\Delta_0$.
\end{proof}

\subsection{The case $d=0$}\label{subsection:d=0}
Let us start with symbols of degree one. We know from Proposition \ref{Cassl} that $C$ vanishes on $\S^0_\delta$ and $\S^1_\delta$ so that the situation is resonant for every $\delta$.
However, by Proposition \ref{gamma1}, the map $\gamma$ vanishes on $\S^1_\delta$ so that $Q_\aff$ defines an equivariant quantization on $\S^1_\delta$. 
Moreover, since $\div X^h = 0$ for all $h \in \Phi(\mathfrak{osp}(p+1,q+1|2r))$, the $\mathfrak{osp}(p+1,q+1|2r)$-modules $\mathcal{S}^1_\delta$ and $\mathcal{S}^1_0$ are equivalent (see formula (\ref{formula:Lie density})). 
Finally, the fact that $\div$ is a $1$-cocycle of the cohomology of ${\rm Vect}(\R^{p+q|2r})$ valued in superfunctions gives
\[
{\rm L}_X (\div S) = \div ({\rm L}_X S)
\]
for every symbol $S$ and every divergence-free vector field $X$. We thus obtain a whole $1$-parameter family of quantizations at order one.
\begin{proposition}
If $d = 0$, then for any $t\in\R$, the map
 \[
 Q \colon \mathcal{S}^1_\delta\to \mathcal{D}^1_{\lambda,\mu}\colon S\mapsto Q_\aff(\Id+t \, \div)(S)
\]
defines a $\mathfrak{osp}(p+1,q+1|2r)$-equivariant quantization over $\R^{p+q|2r}$.
\end{proposition}
For symbols of degree two, it follows from Proposition \ref{minpoly} that the minimal polynomial of $C$ on $\mathcal{S}_{\delta}^{2}$ is given by $(x+4)^{2}$. 
Thus $-4$ is a root of this polynomial with multiplicity two, the restriction of $C$ to this space of symbols is not diagonalizable and we have to use Equation (\ref{eq2}). 
\begin{proposition}
If $d = 0$, the map
\[
Q \colon \mathcal{S}^2_\delta\to \mathcal{D}^2_{\lambda,\mu}\colon S\mapsto Q_\aff(\Id+\frac{1}{2} \, \div)(S)
\]
defines a $\mathfrak{osp}(p+1,q+1|2r)$-equivariant quantization over $\R^{p+q|2r}$.
\end{proposition}

\section{Acknowledgments}

It is a pleasure to thank J.-P.~Michel for stimulating discussion.
P. Mathonet is supported by the University of Luxembourg internal research project F1R-MTH-PUL-09MRDO.
F. Radoux thanks the Belgian FNRS for his research fellowship.


\end{document}